\documentclass[12pt]{article}
\usepackage{amsthm}
\usepackage{amsmath}
\usepackage{cite}
\usepackage{tikz}
\usetikzlibrary{arrows}
\usepackage{amssymb}
\usepackage{enumerate}
\usepackage[margin=1.3in]{geometry}

\usepackage{float}

\newtheorem{theorem}{Theorem}[section]
\newtheorem{conjecture}[theorem]{Conjecture}
\newtheorem{proposition}[theorem]{Proposition}
\newtheorem{corollary}[theorem]{Corollary}
\newtheorem{lemma}[theorem]{Lemma}
\newtheorem{definition}[theorem]{Definition}
\newtheorem{example}[theorem]{Example}
\newtheorem{remark}[theorem]{Remark}
\newtheorem{question}[theorem]{Question}

\include{epsf}
\makeatletter
\newcommand*{\rom}[1]{\expandafter\@slowromancap\romannumeral #1@}
\makeatother

\begin{document}

\title{Cohomological localization for Hamiltonian $S^1$-actions and symmetries of complete intersections}

\author{Nicholas Lindsay}

\maketitle

\begin{abstract} To begin the paper we revisit a cohomological localization result of Jones-Rawnsley which was subsequently improved by Farber, further generalizing the result. We then proceed to improve a previous result of the author on complete intersections of dimension $8k$ with a Hamiltonian $S^1$-action in two directions. Firstly, in dimension $8$ we remove the assumption on the fixed point set. Secondly, in any dimension we prove the result under an analogous assumption on the fixed point set. We also give some applications towards the unimodality of Betti numbers of symplectic manifolds having a Hamiltonian $S^1$-action, and discuss the relation to symplectic rationality problems.
\end{abstract}

\section{Introduction}

This paper is concerned with the topology of closed symplectic manifolds having a Hamiltonian $S^1$-action.  A driving question in symplectic geometry is to understand the extent to which the topology of closed symplectic manifolds differ from that of their
 K\"{a}hler counterparts, of which complex projective varieties are a fundamental
 source of examples. In general the difference is substantial. For instance, the fundamental group of a closed K\"{a}hler manifold satisfies non-trivial restrictions, the most basic being that the abelianization has even rank.  On the other hand every finitely presented group appears as the fundamental group of a closed symplectic $4$-manifold \cite[Theorem 1.1]{Go}. The equivariant version of the problem is an interesting refinement since the
 presence of symmetries may eliminate the difference between the categories. Key examples of such results are given by Karshon's classification of closed symplectic $4$-manifolds with a Hamiltonian $S^1$-action \cite{Ka}, and Tolman's classification of closed symplectic $6$-manifolds having a Hamiltonian $S^1$-action and $b_{2}=1$ \cite{T}. Here classification has a slightly different meaning in each case, we refer the reader to the original papers for the precise results. 

 We start by revisiting a result by Jones and Rawnsley \cite[Theorem 1]{JR}, subsequently improved by Farber \cite[Corollary 6]{F}, about the signature and Betti numbers of a closed symplectic manifold having a Hamiltonian $S^1$-action. These results fit into the above discussion, because they generalize results that were well-known in the complex projective category to the symplectic category under the assumptions of symmetry. See the footnote at the beginning of Section \ref{localization} for a precise statement.

After recalling some preliminaries in Section \ref{prel}, we begin Section \ref{localization} by proving Proposition \ref{general}, a localization formula for the following topological invariant: $$I_{JR}(X) := \sigma(X) - \sum_{i \in \mathbb{Z}} (b_{4i}(X) - b_{4i+2}(X) ). $$ Proposition \ref{general} states that $I_{JR}$ localizes as a signed sum over the fixed point set in a similar way to the signature. This puts the earlier results of \cite[Theorem 1]{JR}, \cite[Corollary 6]{F} into a natural context and generalizes them. The proof of Proposition \ref{general} combines the localization formula for the signature (see Theorem \ref{locsig}), the localization formula for the Betti numbers (see Theorem \ref{locbet}), and  Poincar\'{e} duality, along with some mild arithmetic simplifications. 

In the next part of the article we use Proposition \ref{general} often in combination with other techniques to prove some results which probe the relation between the K\"{a}hler and symplectic categories in the presence of a Hamiltonian $S^1$-action. The main thread is a collection of results regarding complete intersections which we now describe in more detail.

\textbf{Symplectic torus actions on complete intersections.} Subsection \ref{proving} and Section \ref{arbitrary} are dedicated to the following problem.

\begin{question} \label{ques}
Let $(M,\omega)$ be a closed symplectic manifold diffeomorphic to a complete intersection. Under which conditions does $(M,\omega)$ admit a Hamiltonian $S^1$-action?
\end{question}

 To summarize the results relevant to Question \ref{ques}, we set the following notation: Let $X_{m}(d_{1},\ldots,d_{k})$ denote a smooth complete intersection in $\mathbb{CP}^{k+m}$ with multidegree  $(d_{1},\ldots,d_{k})$. Here $m$ and $k$ are the the complex dimension and codimension respectively. For the definition of complete intersection see Definition \ref{cidef}, for more background material about complete intersections see Subsection \ref{completeint}.

 Since complete intersections of dimension at least $4$ are simply connected, there is no distinction between Hamiltonian and symplectic $S^1$-actions. For the sake of brevity, we keep the understanding that a complete intersection $(M,\omega)$ is a closed symplectic manifold diffeomorphic to a complete intersection. For some results appearing later in the article, some more restrictions on the symplectic form are needed and these are stated explicitly. 

Note that in dimension $2$, $S^2$ and $T^2$ have symplectic circle actions with isolated fixed points which are Hamiltonian and non-Hamiltonian respectively, and smooth orientable surfaces of genus at least $2$ do not admit non-trivial smooth actions of $S^1$, this explains the inclusion $X_{1}(1) \cong X_{1}(2)\cong S^2 ,X_{1}(3) \cong X_{1}(2,2) \cong T^2$ in the results below, and the exclusion of all other complete intersection curves. With this understanding, we work with complete intersections of dimension at least four where we may assume by simply connectedness that a symplectic action is Hamiltonian.

The major motivating result of this paper is a theorem of Benoist  \cite[Theorem 3.1]{B}, which is a complete characterization of smooth complete intersections with infinite holomorphic automorphism groups. A comprehensive smooth analogue  in real dimension $6$ was given by Dessai and Wiemeler \cite{DW}. One of the main aims of this article is to obtain a symplectic analogue of this theorem, specifically Theorem \ref{compin}, Corollary \ref{completeiso} and Theorem \ref{fs}. 

The first main result is to remove the technical assumption on the fixed point set in our previous result on complete intersections \cite[Theorem 1.1]{L} in dimension 8. 

\begin{theorem} \label{compin}
Suppose that $(M,\omega)$ is an $8$-dimensional complete intersection with a symplectic $S^1$-action. Then $M$ is diffeomorphic to one of $X_{4}(1),X_{4}(2)$ or $X_{4}(2,2)$.
\end{theorem}

In fact we will deduce this result from the more general Theorem \ref{main} which states that a closed symplectic $8$-manifold $M$ having a Hamiltonian $S^1$-action and satisfying $b_{2}(M)=1$ has positive definite intersection form. The proof is given in Subsection \ref{proving}. The first step of the proof is to note that the condition $I_{JR}(M)=0$ is equivalent to the intersection form being positive definite. Then, since $\dim(M)=8$ all the terms of the right hand side of Equation (\ref{equa}) of Proposition \ref{general} are of the form $\pm I_{JR}(F)$ where $F$ is a $4$-dimensional fixed component. The proof is concluded by proving that $I_{JR}(F) =0$ for all $4$-dimensional fixed components. This is proven in a different way depending on whether $F$ is extremal or not, and whether $F$ is contained in an isotropy $6$-manifold. The most difficult case when $F$ is non-extremal and not contained in an isotropy $6$-manifold, the desired condition is proved using the Duistermaat-Heckman theorem and a result from \cite{OO}.

 Theorem \ref{main} may be considered as a result regarding a certain class of monotone symplectic $8$-manifolds with a Hamiltonian $S^1$-action, see Subsection \ref{proving} for more precise details. There has been a lot of recent progress towards classifying monotone symplectic manifolds with Hamiltonian symmetries, see \cite{C2,CK3,GVS,LP1,SS} and the references contained therein.

In Section \ref{arbitrary}, we prove a classification result for symplectic $S^1$-actions on complete intersections under a condition on the fixed point set analogous to \cite[Theorem 1.1]{L}, therefore removing the assumption on the dimension in \textit{op. cit}.

\begin{theorem} \label{completeiso}
Suppose that $(M,\omega)$ is a complete intersection admitting a symplectic $S^1$-action with no fixed component $F$ satisfying $\dim(F)>0$ and $\dim(M) = \dim(F) \mod 4$. Then $M$ is diffeomorphic to one of $X_{m}(1), X_{m}(2), X_{2m}(2,2),$  for some $m \in \mathbb{N}$ or $X_{1}(3)$, $X_{2}(3)$. 
\end{theorem}

The proof of Theorem \ref{completeiso} for $n$ even is similar to \cite[Theorem 1.1]{L}, applying Proposition \ref{general} to obtain $I_{JR}(M)=0$  and then applying the classification of Libgober and Wood  \cite[Corollary 6.1]{LW}. For $n$ odd the proof is different, relying on the localization of Betti numbers formula for Hamiltonian $S^1$-actions (Theorem \ref{locbet}), known properties about the Betti numbers of complete intersections (Lemma \ref{betti}), then finally applying the classification of complete intersections having vanishing middle Betti number due to Ewing and Moolgavkar \cite[Corollary 5]{EM}.

Theorem \ref{completeiso} has the following consequence.

\begin{corollary} \label{iso}
Suppose that $(M,\omega)$ is a complete intersection admitting a symplectic $S^1$-action having only isolated fixed points. Then $M$ is diffeomorphic to one of $X_{m}(1), X_{m}(2), X_{2m}(2,2),$  for some $m \in \mathbb{N}$ or $X_{1}(3)$, $X_{2}(3)$. 
\end{corollary}

Projective spaces and quadrics of any dimension admit a Hamiltonian torus action with only isolated fixed points. Each of the smooth $4$-manifolds underlying $X_{2}(2,2)$ and $X_{2}(3),$ admits a symplectic form invariant under a torus action, since either is diffeomorphic to $\mathbb{CP}^2$ blown-up in finitely many points. Hence, the only unknown case remaining from Theorem \ref{iso} is $X_{2m}(2,2)$ for $m \geq 2$.

 For a complete intersection $M$ let $\omega_{FS}$ denote the restriction of the Fubini-Study form to $M$. In Subsection \ref{GKM}, we are able to deal with $X_{2m}(2,2)$ when the action is GKM, and the symplectic form is homotopic to $\omega_{FS}$. Thus obtaining in Theorem \ref{fs} a sharp classification of the possible multidegrees of complete intersections admitting such actions, as in Benoist's original result on automorphism groups.

\begin{theorem} \label{fs}
Let $M$ be a complete intersection. Then $(M,\omega_{FS})$ has a GKM symplectic torus action $\iff  \;\; |Aut(M)| = \infty$.
\end{theorem}

 Supposing the existence of such an action for a contradiction on $X_{2n}(2,2)$, the proof uses \cite[Corollary 3.1]{GS} and \cite[Proposition 3.4]{LW} to obtain an expression for $c_{1}c_{n-1}$ in terms of $n$ and \cite[Lemma 4.3]{GVS}, to evaluate the integral of $c_1$ along a hypothetical toric $1$-skeleton Poincar\'{e} dual to $c_{n-1}$. This allows us to obtain an inequality contradicting the aforementioned expression for $c_{1}c_{n-1}$.

\textbf{Unimodality of Betti numbers of closed symplectic manifolds admitting a Hamiltonian $S^1$-action.} In the final section of the paper, we prove some new results regarding the unimodality of Betti numbers of a closed symplectic manifold admitting a Hamiltonian $S^1$-action. Tolman has asked whether for a closed symplectic manifold with a Hamiltonian $S^1$-action, the even Betti numbers are unimodal (i.e. weakly increasing up to the middle degree) \cite[Problem 4.3]{JHKLM}. For K\"{a}hler manifolds the property follows from the hard Lefschetz theorem.  In \cite{CK} Cho and Kim gave an affirmative answer to \cite[Problem 4.3 (1)]{JHKLM} when $M$ is $8$-dimensional and the action has isolated fixed points. In Theorem  \ref{eightunimodal} we prove the unimodality of even Betti numbers for $8$-dimensional manifolds having a Hamiltonian $S^1$-action with no extremal $4$-dimensional fixed components.

\begin{theorem} \label{eightunimodal}
Suppose that $(M,\omega)$ is a closed symplectic $8$-manifold having a Hamiltonian $S^1$-action, and neither $M_{\min}$ nor $M_{\max}$ is $4$-dimensional. Then $b_{2}(M) \leq b_{4}(M)$.
\end{theorem}

 In Proposition \ref{twelve} we prove that for a closed symplectic $12$-manifold with $b_{2}(M)=1$, satisfying a similar assumption on the fixed point set, the even Betti numbers are unimodal. To the author`s knowledge this is new even under the assumption of isolated fixed points. Furthermore, in Proposition \ref{ineq} we give weaker versions of Proposition \ref{twelve} in general dimensions, generalizing an inequality of Cho \cite{C}  to more general fixed point sets. Finally, in Proposition \ref{roapp} we combine this result with the Rokhlin-Ochanine theorem to obtain a final application.

\textbf{Acknowledgments.} I would like to thank Leonor Godinho and Silvia Sabatini for helpful discussions, in particular explaining to me the results of \cite{GVS}, which allowed me to deal with the intersection of two quadrics in the GKM case. I presented some of the results about unimodality of Betti numbers in the workshop ``Perspectives in equivariant topology" in Cologne September 2022, I would like the thank those present for interesting discussions, in particular Liat Kessler and Susan Tolman.  I would like to thank Dmitri Panov for collaborations and discussions which influenced the overall approach of the project and for comments on a draft version. I would like thank the referee for the detailed report which greatly improved the article.

The author was supported by SFB-TRR 191 grant Symplectic Structures in Geometry, Algebra and Dynamics funded by the Deutsche Forschungsgemeinschaft.

 \section{Preliminaries} \label{prel}
 \subsection{Hamiltonian $S^1$-actions}  \label{hamact}

 Let $\mathfrak{t}$ denote the Lie algebra of the compact torus $T$, and $\mathfrak{t}^*$ its dual. Recall that if a manifold $M$ has a smooth $T$-action, then for every $t \in \mathfrak{t}$ we can naturally associate an invariant vector field on $M$; $X_{t}$, tangent to the orbits of the action of the one-parameter family associated to $t$.

\begin{definition}
Let $(M,\omega)$ be a symplectic manifold with a smooth $T$-action by symplectomorphisms. The action is called \textbf{Hamiltonian} if there is a smooth $T$-invariant function $\mu: M \rightarrow \mathfrak{t}^*$ called the \textbf{moment map} such that $$\omega (X_t, \cdot) = -d(\mu(t)),$$ for all $t \in  \mathfrak{t}$.
\end{definition}

 The fixed point set of the action, denoted $M^{T}$  is equal to the critical set of $\mu$, and is a collection of symplectic submanifolds \cite[Theorem 5.47]{MS}. A connected component of $M^T$ is called a \textbf{fixed component}. Moreover, there always exists a $T$-invariant, compatible almost complex structure.
 
When $T$ is $1$-dimensional, the moment map is referred to as the Hamiltonian and denoted $H: M \rightarrow \mathbb{R}$. The Hamiltonian $H$ is a perfect Morse-Bott function (\cite[Page 34]{Ki}, see Theorem \ref{locbet} below for a full statement), and a perfect Morse function if and only if $M^{S^1}$ is finite. 

Let $G \subset S^1$ be a subgroup. The subset $M^{G} = \{x \in M | \;g.p=p \; \forall g \in G\}$ is an $S^1$-invariant symplectic submanifold for which the $S^1$-action is Hamiltonian, with Hamiltonian $H|_{M^{G}}$ \cite[Lemma 5.53]{MS}. This will be used in the proof of Theorem \ref{compin}. Connected components of $M^{G}$ are called \textbf{isotropy submanifolds}.

Let $J$ denote a $T$-invariant, compatible almost complex structure. Note that at $p \in M^{S^1}$, $T_{p}M$ admits the structure of a complex $S^1$-representation. Therefore, we may decompose $T_{p}M$ into irreducible representations:

$$z.(z_{1},\ldots, z_{n}) = (z^{w_1}z_1 , \ldots, z^{w_n}z_n ),$$ where $w_{i} \in \mathbb{Z}$ for each $i$. The multiset of integers $\{w_i\}$ is an invariant of the representation and called the \textbf{weights} of the $S^1$-action at $p$. 

 For a fixed component $F$ denote by $\lambda_{F}$  half the Morse-Bott index of $H$ along $F$. It holds that $\lambda_{F}$ is also equal to the the number of strictly negative weights of the $S^1$-action along $F$ counted with multiplicity \cite[Page 7]{T}. We will need the following localization formula for the Betti numbers due to Kirwan \cite[Page 34]{Ki}.  We note that in \cite{Ki} the result appears as an auxiliary result to prove that the norm squared of the moment map for a general Hamiltonian Lie group action is perfect, see \cite[Section 2]{T} for a comprehensive treatment focusing only on material relevant to proving Theorem \ref{locbet} for $S^1$-actions.
 
 \begin{theorem} \label{locbet} \cite[Page 34]{Ki}
 Let $(M,\omega)$ be a closed symplectic manifold having a Hamiltonian $S^1$-action. Then for each $i$, $$ b_{i}(M) = \sum_{F \subset M^{S^1}} b_{i - 2\lambda_{F}} (F) ,$$ where the sum runs over each of the fixed components. 
 \end{theorem}

We use $\sigma(M)$ to denote the signature of $M$, which is equal to the signature of the symmetric bilinear form $$H^{\frac{\dim(M)}{2}}(M,\mathbb{R}) \times H^{\frac{\dim(M)}{2}}(M,\mathbb{R}) \rightarrow \mathbb{R},$$ that is $$\sigma(M) = b_{\frac{\dim(M)}{2}}^{+}(M) - b_{\frac{\dim(M)}{2}}^{-}(M), $$ where $b_{\frac{\dim(M)}{2}}^{+}(M)$ (resp. $b_{\frac{\dim(M)}{2}}^{-}(M)$) denotes the maximal subspace on which the form is positive (resp. negative) definite.  $\sigma(M)$ is defined to be equal to zero for manifolds of dimension not divisible by $4$ and is equal to $1$ when $M$ is a single point.

The following localization formula is a consequence of the equivariant version of the Atiyah-Singer index theorem. The formula we will use is derived in \cite[Section 5.8]{HBJ}. 
  
  \begin{theorem} \label{locsig}
 Let $(M,\omega)$ be a closed symplectic manifold having a Hamiltonian $S^1$-action. Then $$\sigma(M) = \sum_{F \subset M^{S^1}} (-1)^{\lambda_{F}} \sigma(F) .$$ 
 \end{theorem}

\subsection{GKM actions}
GKM actions introduced in \cite{GKM} are a natural class of torus actions in various categories, which having associated (multi)graph, usually called the GKM graph encoding topological and geometric information about the torus action. The amenability of the GKM graph to combinatorial arguments make GKM actions an attractive class of torus actions to study. In \cite{GVS} results about GKM actions in the setting Hamiltonian torus actions on closed symplectic manifolds were proved, which we will use in Subsection \ref{GKM}.  Firstly, we recall the definition. 

\begin{definition}
Let $(M,\omega)$ be a closed symplectic manifold with a Hamiltonian $T^k$-action. The action is called GKM if for every codimension $1$ subgroup $K \subset T^k$, $M^{K}$ has dimension at most $2$.
\end{definition}

If such manifold has a GKM action then the subsets $M^{K}$ for each codimension $1$ subgroup $k$ form a collection of $\frac{n \chi(M)}{2}$, invariant symplectic $2$-spheres, corresponding to the edges of the GKM graph, whose union is called the \textbf{toric 1 skeleton}. To prove Theorem \ref{GKM} we will need \cite[Lemma 4.3]{GVS}, we restate it here for convenience.

\begin{lemma} \label{torone} \cite[Lemma 4.3]{GVS}
Let $(M,\omega)$ be a closed symplectic $2n$-manifold with a Hamiltonian GKM $T^k$-action. Then the toric $1$-skeleton is Poincar\'{e} dual to $c_{n-1}$.
\end{lemma}

\subsection{Applications of Duistermaat-Heckman} \label{apdh}

We begin by recalling the Duistermaat-Heckman theorem  \cite[Theorem 1.1]{DH}. Suppose that $(M,\omega)$ is a closed symplectic manifold with a Hamiltonian $S^1$-action, and let $c$ be a regular value of $G$. Recall that $$M_{c} := H^{-1}(c)/S^1$$ has the structure of a symplectic orbifold, induced via projection from the coisotropic submanifold $H^{-1}(c) \subset M$, see \cite[Page 175]{MS}.

Before stating the Duistermaat-Heckman theorem, we have to recall some of the properties of the gradient flow for a Hamiltonian $S^1$-action. Fix an $S^1$-invariant compatible almost complex structure $J$, and let $g$ be the associated metric. Recall gradient vector field $\nabla_{g}H$, is the vector field $g$-dual to the $1$-form $dH$. Then, if $(a,b)$ is an interval on which $H$ is regular and $c,d \in (a,b)$ the flow of $Y$ gives an $S^1$-equivariant diffeomorphism $H^{-1}(a) \rightarrow H^{-1}(b)$ and therefore a orbifold diffeomorphism $F_{a,b} :M_{a} \rightarrow M_{b}$. The Duistermaat-Heckman theorem  \cite[Theorem 1.1]{DH} states that the variation of the cohomology $$F_{a,b}^{*}[\omega_{b}] \in H^{2}(M_{a},\mathbb{R}) $$ where $b$ varies and $a$ is constant is linear and gives an explicit formula for the direction of this linear path. As is usual in the subject, in the statement and applications of the Duistermaat-Heckman theorem that follow, we keep the above understanding of what variation of $[\omega_{c}]$ means, and do not make explicit reference to the gradient flow.

\begin{theorem} \label{dhthe} \cite[Theorem 1.1]{DH}
Suppose that $(M,\omega)$ is a closed symplectic manifold with a Hamiltonian $S^1$-action, with Hamiltonian $H$. Let $c$ be a regular value of $H$. Then $$\frac{d}{dt}[\omega_t] |_{t=c}= e(H^{-1}(c)) .$$ 

Here $H^{-1}(c)$ is considered as an (orbifold) principal $S^1$-bundle over the reduced symplectic orbifold $M_{c}$. Finally, $e$ denotes the Euler class of this bundle in $H^{2}(M_{t},\mathbb{R})$.
\end{theorem}

It is known that when $(M,\omega)$ is a closed symplectic manifold with a Hamiltonian $S^1$-action, all level sets of the Hamiltonian $H$  are connected  \cite[Lemma 5.51]{MS}. We denote by $M_{\min}$ and $M_{\max}$ the subsets on which the Hamiltonian obtains its minimum and maximum respectively. Note that since $H$ is critical along these subsets, they are pointwise fixed by $S^1$. Moreover, they are both connected, symplectic submanifolds by \cite[Lemma 5.51]{MS} and  \cite[Theorem 5.47]{MS} respectively. The following lemma was proven in \cite{T} using the Duistermaat-Heckman Theorem.

\begin{lemma} \label{hamiltonianlemma}\cite[Lemma 3.1]{T} Let $(M,\omega)$ be closed symplectic manifold with a Hamiltonian $S^1$-action. Then for a fixed component $F$, $$\lambda_{F} \leq \sum_{ H(F') < H(F)} (\frac{\dim(F')}{2} + 1). $$
\end{lemma}

We call a fixed point component $Y$ internal if $H(M_{\min}) < H(Y) < H(M_{\max})$. Lemma \ref{hamiltonianlemma} has the following consequence which we will use.

\begin{lemma} \label{DH Lemma}
Suppose that $(M,\omega)$ is a closed symplectic manifold with a Hamiltonian $S^1$-action. Suppose that $M_{\min}$ is an isolated fixed.  If $Y$ is an internal fixed point component such that $H(Y') \geq H(Y)$ for any other internal fixed component $Y'$, then $\lambda_{Y} = 1$.  If $Y$ is an internal fixed point component such that $H(Y') \leq H(Y)$ for any other internal fixed component $Y'$, then $\lambda_{Y} = n-1$.
\end{lemma}
\begin{proof}
The first claim follows from applying Lemma \ref{hamiltonianlemma}. The second statement follows from applying Lemma \ref{hamiltonianlemma} to the Hamiltonian $S^1$-action with generating vector field $-X$ and Hamiltonian $-H$.\end{proof}

A theorem of Ono states that is closed symplectic manifold $(M,\omega)$ satisfies $c_1 = k [\omega]$ $k \in \mathbb{R}$ then if $k \leq 0$ $(M,\omega)$ cannot have a non-trivial Hamiltonian $S^1$-action \cite[Proposition 1.1]{On}.  It was also proved that if $k<0$ that $(M,\omega)$ cannot have a symplectic $S^1$-action \cite[Theorem 1]{On}.

\begin{lemma}  \label{symfano}  Suppose that $(M,\omega)$ is a closed symplectic manifold with a Hamiltonian $S^1$-action and $b_{2}(M)=1$. Then after possibly rescaling the symplectic form by a positive constant we may assume that $c_{1}(M)=[\omega]$.
\end{lemma}
\begin{proof}
Since $b_{2}(M)=1$, $c_{1}(M) = k [\omega]$ for some $k \in \mathbb{R}$. Therefore the result follows immediately from \cite[Proposition 1.1]{On}.\end{proof}

For Hamiltonian $S^1$-actions on Fano varieties, the following normalization for the Hamiltonian goes back to Futaki \cite[Theorem 1]{Fu}. In the case of monotone symplectic manifolds proofs may be found for the semi-free case in \cite{CK2} and for the general case in \cite{GVS}. \footnote{ In \cite[Proposition 1.9]{LP1} the formula was proven for closed symplectic manifolds with Hamiltonian $S^1$-actions, such that $c_1$ and $[\omega]$ have the same evaluation on classes in the image of the Hurewicz map $\pi_{2}(M) \rightarrow H_{2}(M)$, this will not be needed in this paper.}

\begin{theorem} \label{wsf}
Suppose that $(M,\omega)$ is a closed symplectic manifold with $c_{1}(M) = [\omega]$ with a Hamiltonian $S^1$-action with Hamiltonian $H: M \rightarrow \mathbb{R}$. Then, after possibly adding a constant to the Hamiltonian, for each fixed component $F$, $$ H(F) = -\sum_{i=1}^{n} w_i ,$$ where $w_i$ are the weights along $F$.
\end{theorem}

The following theorem is due to Ohta and Ono, and builds on fundamental work of Gromov, McDuff and Taubes. 

\begin{theorem} \label{ooth} \cite{OO}
Let $(N,\omega)$ be a closed symplectic $4$-manifold such that $[\omega] = k c_{1}(N)$ for some $k > 0$. Then, $N$ is diffeomorphic to a del Pezzo surface. 
\end{theorem}

The theorem has the following immediate consequence.

\begin{corollary} \label{oocor}

Let $(N,\omega)$ be a closed symplectic $4$-manifold such that $[\omega] = k c_{1}(N)$ for some $k > 0$. Then, $b_{2}^{+}(N)=1$.
\end{corollary}

\begin{proof}
By Theorem \ref{ooth}, $N$ is diffeomorphic to a smooth del Pezzo surface. By the classification of del Pezzo surfaces, $N$ is either diffeomorphic to $\mathbb{CP}^1 \times \mathbb{CP}^1 $ or $\mathbb{CP}^2 \# k\overline{\mathbb{CP}^2}$, where $0 \leq k \leq 8$. One may verify by direct calculation that each of these manifolds satisfies  $b_{2}^{+}=1$.\end{proof}
\subsection{Complete Intersections} \label{completeint}
We start this subsection by recalling a basic definition of algebraic geometry. Let  $X \subset \mathbb{CP}^{n+k}$ be a smooth subvariety of complex dimension $n$.

\begin{definition}\label{cidef} $X$ is called a \textbf{complete intersection} if it is a  transversal intersection of $k$ smooth hypersurfaces. \end{definition} Since we restrict to the smooth case transversality can be taken simply in the sense of complex submanifolds, but also has algebraic reformalizations that are useful; such as in terms of gradients or the ideal associated to $X$.  Let $\{d_{i}|\;i=1,\ldots, k \}$ be the multiset consisting of the degrees of the $k$ hypersurfaces. Let $X_{m}(d_{1},\ldots,d_{k})$ denote a smooth complete intersection in $\mathbb{CP}^{k+m}$ with multidegree  $(d_{1},\ldots,d_{k})$.  A well-known consequence of Ehresmann's Theorem is that the diffeomorphism type of $X_{m}(d_{1},\ldots,d_{k})$ depends only on the dimension and the multidegree. Hence as a smooth manifold the notation $X_{m}(d_{1},\ldots,d_{k})$ is unambiguous. Note however, that complete intersections with different multidegree can sometimes be diffeomorphic, a simple example being furnished by $X_{1}(3) \cong X_{1}(2,2) \cong S^1 \times S^1$. Another basic fact following from the adjunction formula is that a complete intersection is Fano resp. generalized Calabi-Yau resp. general type, when $\sum_{i} d_i$ is less than resp. equal to resp. greater than $n+1$.

By applying the Lefschetz hyperplane theorem consecutively on the image of the variety via appropriate Veronese embeddings, the following holds:
\begin{lemma} \label{betti}  Let $X$ be a smooth complete intersection of complex dimension $n$, then $b_{i}(X) = b_{i}(\mathbb{CP}^{n}) \text{ for }i \neq n.$
\end{lemma} The topological consequences of this fact have a different flavor when $n$ is odd or even. In both cases, there is an important characterization of when a complete intersection is ``close to projective space", which is due to Ewing and Moolgavkar  when $n$ is odd \cite[Corollary 5]{EM} and Libgober and Wood  when $n$ is even  \cite[Corollary 6.1]{LW}.

In the case when $n$ is odd, the middle degree cohomology $H^{n}(M,\mathbb{R})$ has the structure of a symplectic vector space given by the intersection product and does not contain cup products of cohomology classes of other degrees. Thus the ``odd part" of the cohomology ring of $M$ is essentially determined by $b_{n}(M) \in 2 \mathbb{Z}$. A complete characterization of when this part of the cohomology ring vanishes, as it does for projective space, was given by Ewing and Moolgavkar \cite{EM}. 

\begin{theorem} \label{betichi} \cite[Corollary 5]{EM}
Let  $X$ be a complete intersection with odd complex dimension $n$, then $\chi(M) = n+1$ if and only if $X$ is diffeomorphic to  $X_{n}(1)$ or $X_{n}(2)$.
\end{theorem}
 
In the case when $n$ is even, then the intersection form on $H^{n}(M,\mathbb{R})$ is symmetric, and $H^{n}(M,\mathbb{R})$ plays a more interesting role in determining the cohomology ring of $M$. The most fundamental invariant we can extract from the intersection form is the signature $\sigma(M)$. A property satisfied by projective space is that the signature is equal to the alternating sum of even Betti numbers, as was shown by Libgober and Wood, this almost gives a characterization of projective space among complete intersections. 

\begin{theorem} \cite[Corollary 6.1]{LW} \label{lw}
Let $X$ be a complete intersection with even complex dimension $n$. Then $$\sigma(X) = \sum_{k \in \mathbb{Z}}( b_{4k}(X) - b_{4k+2}(X))$$ if and only if $X$ is diffeomorphic to $X_{n}(1), X_{n}(2), X_{2}(3)$ or $X_{n}(2,2)$.  That is: \begin{itemize}
\item[a.)] If $4|n$, the intersection form is positive definite if and only $X$ is diffeomorphic to $X_{n}(1), X_{n}(2)$ or $X_{n}(2,2)$

\item[b.)] If $n = 2 \mod 4$, $b_{n}^{+}(X)$=1 if and only $X$ is diffeomorphic to $X_{n}(1), X_{n}(2), X_{n}(2,2)$ or $X_{2}(3)$
\end{itemize}
\end{theorem} 

The proof of the equivalence of a.)-b.) with the first statement, follows from writing $\sigma(X) = b_{n}^{+}(X)-b_{n}^{-}(X)$, $b_{n}(X) = b_{n}^{+}(X)+b_{n}^{-}(X)$ and then applying Lemma \ref{betti}.

\begin{remark} A beautiful fact, proved by Deligne  \cite[Page 339]{D} is that when $4|n$ the intersection form of $X_{n}(2,2)$ is the  Korkine-Zolotareff form $\Gamma_{n+4}$. That is the lattice of $\mathbb{R}^{n+4}$ generated by $e_{i}+e_{j}$ and $\frac{1}{2}(e_{1} + \ldots + e_{n+4}).$ In particular for $n=4$ the intersection form is isomorphic to the $E_{8}$ lattice. Deligne gives a Dynkin diagram description of the intersection form for $n = 2 \mod 4$, and it is clear from this description that $b^{+}_{n}(X_{n}(2,2)) = 1$. \end{remark}

Finally, we state one of the major motivating results of this article due to Benoist. Let $X$ be a smooth projective variety, by a well-known relative version of Chow's theorem the group of bihomomorphisms of $X$ coincides with the group of isomorphisms of $X$, we denote it by $Aut(X)$.

\begin{theorem}  \cite[Theorem 3.1]{B} \label{benoist}
Let $X$ be a complex smooth complete intersection of dimension at least two. Then $|Aut(X)| = \infty$ if and only $X$ is isomorphic to $X_{n}(1)$ or $X_{n}(2)$.
\end{theorem}

There are two remarks to make here. In dimension $1$ a smooth elliptic curve (realized either a cubic plane curve or an intersection of two quadrics in $\mathbb{P}^3$) has infinite automorphism group hence the omission of dimension $1$ from Benoist's theorem. Secondly, as was noted in the introduction, the smooth manifolds underlying complete intersections $X_{2}(2,2)$ and $X_{2}(3)$ do admit algebraic structures with infinite automorphism groups, however these algebraic structure cannot be isomorphic to a complete intersection. To see this, note by the adjuction formula such complete intersections are Del Pezzo, of degree $3$ and $4$ respectively, the finiteness of the automorphism groups of such surfaces was known classically.

\section{Proof of Theorem \ref{compin}} \label{localization}

We now proceed to the proof of our main result Theorem \ref{main}. Before giving the details of the proof, we describe the general approach. 
 Firstly we prove Proposition \ref{general} which is a localization formula for a certain topological invariant $I_{JR}$ which is defined below. The idea of Proposition \ref{general} proposition was motivated by a theorem of Jones and Rawnsley  \cite[Theorem 1.1]{JR},\footnote{ \label{jrprojective}
We note that \cite[Theorem 1]{JR} was known much earlier for smooth complex projective varieties having an algebraic torus action with only isolated fixed points. It may be proved by combining the Hodge index theorem, and the fact due to Białynicki-Birula that the Hodge numbers such a variety are pure i.e. $h^{ij}=0$ for $i \neq j$ \cite{BB}. In this sense \cite[Theorem 1]{JR} may be considered as a symplectic version of the Hodge index theorem. 
} which states that if a closed symplectic $(M,\omega)$ has a Hamiltonian $S^1$-action with isolated fixed points then it satisfies the JR equation.
 It is necessary to make the following definition. 

\begin{definition}
Suppose that $X$ is a closed manifold. The \textbf{JR equation} is $$I_{JR}(X) := \sigma(X) - \sum_{i \in \mathbb{Z}} (b_{4i}(X) - b_{4i+2}(X) ) = 0.$$ 
\end{definition}

In Proposition \ref{general} below, it is shown that the invariant $I_{JR}$ localizes in a simple way over the fixed point set. 
\begin{proposition}\label{general}
Suppose that $(M,\omega)$ is a closed symplectic manifold admitting a Hamiltonian $S^1$-action. Then, \begin{equation} \label{equa} I_{JR}(M) = \sum_{\{ F \subset M^{S^1}| \dim(F)=4k>0 \}} (-1)^{\lambda_{F}} I_{JR}(F) . \end{equation}
\end{proposition}

\begin{proof}

By Theorems \ref{locbet}  and \ref{locsig} the alternating sum of even Betti numbers and the signature may be be expressed as sums over the fixed points set, in a similar way to the proof of \cite[Proposition 3.1]{L}. By Theorem \ref{locbet}:

$$ \sum_{j \in \mathbb{Z}}b_{4j}(M) - \sum_{j \in \mathbb{Z}} b_{4j +2}(M)  =  \sum_{F \subset M^{S^1}}( \sum_{j \in \mathbb{Z}} b_{4j -2\lambda_{F}}(F) - \sum_{j \in \mathbb{Z}} b_{4j +2 - 2\lambda_{F}}(F)).$$

For each component $F$, by shifting the indices in the above alternating sum, we obtain that:

$$(\sum_{j \in \mathbb{Z}} b_{4j - 2\lambda_{F}}(F) - \sum_{j \in \mathbb{Z}} b_{4j +2 - 2\lambda_{F}}(F) )  = (-1)^{\lambda_F} ( \sum_{j \in \mathbb{Z}} b_{4j }(F) - \sum_{j \in \mathbb{Z}} b_{4j +2 }(F) ). $$

By Theorem \ref{locsig} it holds that: $$ \sigma(M) =  \sum_{F \subset M^{S^1}} (-1)^{\lambda_{F}} \sigma(F)  . $$

Hence,

$$I_{JR}(M) = \sum_{F \subset M^{S^1}} (-1)^{\lambda_{F}} I_{JR}(F) . $$ 

It remains to show that if $F$ is an isolated fixed point or $\dim(F) = 2 \mod 4$, then $I_{JR}(F)=0$. For $F$ an isolated fixed point, the statement is straightforward to prove, recalling that a point has signature $1$.

Suppose $N$ has dimension $2 \mod 4$, by definition $\sigma(N)=0$. Hence, it suffices to prove $I_{JR}(N)=0$ it is sufficient to prove $$\sum_{i \in \mathbb{Z}}( b_{4i}(N) - b_{4i+2}(N) ) = 0. $$ This follows from Poincar\`{e} duality and the assumption that $\dim(N)= 2 \mod 4$.\end{proof}

\begin{remark}
A form of Proposition \ref{general} with additional conditions has been proven by Farber \cite[Corollary 6]{F}. Farber's result states that for a manifold with a Hamiltonian $S^1$-action such that $
I_{JR}(F) = 0$, $ \forall F \subset M^{S^1}$ and all the odd Betti numbers of each $F$ vanish, it holds that $I_{JR}(M) = 0$. Proposition \ref{general} improves on  \cite[Corollary 6]{F} in three ways. Firstly, it removes the assumption on the odd Betti numbers. Secondly, it shows that only fixed submanifolds of positive dimension divisible by $4$ are relevant. Lastly, the equation of Proposition \ref{general} can be useful for an inequality analysis when both sides are not necessarily zero, this is key for the proof of Theorem \ref{eightunimodal}.
\end{remark}

Given the generality of Proposition \ref{general}, it is natural ask whether $I_{JR}(M)=0$ for all closed symplectic manifolds admitting a Hamiltonian $S^1$-action. The following example shows that this is not the case.

\begin{example} Consider the standard Hamiltonian $S^1$-action on $\mathbb{CP}^4$ with the Fubini-Study form,  given by $$\lambda \cdot [z_0:\ldots:z_{4}] =[\lambda z_0:z_1:\ldots:z_{4}]  .$$ Then, $(\mathbb{CP}^4)^{S^1}$ has two components $[1:0:\ldots:0]$ and the hyperplane $H$ defined by the equation $z_{0}=0$. Let $S \subset H$ be a smooth degree $4$ hypersurface. Recall that $S$ is a K3 surface with signature $16$ and $b_{2}(S)=22$. Let $M$ be the $S^1$-equivariant blow-up of $\mathbb{CP}^4$ along $S$. Then $M^{S^1} $ has three components diffeomorphic to $H$, $S$ and point respectively, moreover the fixed point component diffeomorphic to $S$ satisfies $\lambda_{S}=1$. By Proposition \ref{general} $$I_{JR}(M) = -I_{JR}(S) = -36. $$
\end{example}

The following result simplifies $I_{JR}$ for $4$-dimensional components.

\begin{lemma} \label{dlemma}
Suppose that $N$ is a closed manifold of dimension $4$. Then  $$ I_{JR}(N) = 2b_{2}^{+}(N)-2.$$
\end{lemma}
\begin{proof}

Recall that $ b_{2}(N)  = b_{2}^{+}(N) + b_{2}^{-}(N)$, $ \sigma(N) = b_{2}^{+}(N) - b_{2}^{-}(N)$ and  that $b_{0}(N) = b_{4}(N)=1$.
So it follows that 
$$I_{JR}(N) =  (b_{2}^{+}(N) - b_{2}^{-}(N)) - 2 + (b_{2}^{+}(N) + b_{2}^{-}(N)) =  2b_{2}^{+}(N)-2,$$ proving the lemma. \end{proof}

\subsection{Proving the JR equation for $4$-dimensional fixed components} \label{proving}

In the theorem below, we complete the main step to classifying $8$-dimensional complete intersections with a symplectic $S^1$-action. The most important result which motivates Theorem \ref{main} is a result of Tolman \cite{T}, which classified Hamiltonian $S^1$-actions on $6$-manifolds with $b_{2}=1$. In \cite{GLS}, building on results of \cite{GVS}, substantial progress towards the problem of classifying $8$-dimensional symplectic manifolds having a Hamiltonian torus action with isolated fixed points and $b_{2}=1$ was completed. A sharp classification result into four explicit cases was obtained when the torus is two dimensional and $b_{4}=2$ or, equivalently, when the action has exactly $6$ fixed points \cite[Main Theorem]{GLS}. All of the examples have various subgroup $S^1$-actions with various fixed components of positive dimension.

\begin{theorem} \label{main}
Suppose that $(M,\omega)$ is a closed symplectic $8$-manifold admitting a Hamiltonian $S^1$-action, and that $b_{2}(M)=1$. Then the intersection form of $M$ is positive definite.
\end{theorem} 

\begin{proof}
Firstly we note that the conclusion of the theorem is equivalent to showing that $I_{JR}(M)=0$ (see Theorem \ref{lw}). Moreover, by Theorem \ref{general} and Lemma \ref{dlemma} if $b_{2}^{+}(N)=1$ for each $4$-dimensional fixed point component $N$, then the result holds.

Let $N$ be a $4$-dimensional fixed component, by a dimension count it follows that $0 \leq \lambda_N \leq 2$. Suppose that $\lambda_N \in  \{0,2 \}$, then by Theorem \ref{locbet} $b_{2}(N)=1$. Since $N$ is a closed symplectic submanifold $b_{2}(N)>0$.   In the localization sum for the Betti numbers Theorem \ref{locbet}, $b_{2}(N)$ either contributes to $b_{2}(M)$ or $b_{6}(M)$ (when $\lambda_{N}=0$ or $2$ respectively). By hypothesis and Poincar\'e duality we have that $b_{2}(M) = b_{6}(M) = 1$. It follows that $b_{2}(N)=1$ in each case. Finally, since $b_{2}(N)=1$ and $N$ is a closed symplectic submanifold $\int_{N} [\omega_N]^2 >0$ implying that $b_{2}^{+}(N)=1$.

Hence we may assume that $\lambda_{N} = 1$, this case is more involved. Firstly this assumption implies by Theorem \ref{locbet} that $M_{\min}$ and $M_{\max}$ are isolated fixed points. 

By Theorem \ref{locbet} every other fixed component is an isolated fixed point $p$ with $\lambda_p = 2$. By Lemma \ref{DH Lemma} that such fixed points do not occur. More precisely if such fixed points satisfy $H(p) \leq H(N)$ then they are excluded by the first part of Lemma \ref{DH Lemma} and if $H(p) \geq H(N)$ they are excluded by the second part.

Hence, there are three fixed components $M_{\min}$, $M_{\max}$ and $N$. By Lemma \ref{symfano} we may assume that $c_{1}(M) = [\omega]$. Since $\dim(N)=4$ and $\lambda_{N} = 1$ the normal bundle of $N$ splits as a sum of two line bundles, one upward flowing with respect to the Hamiltonian, and one downward flowing. We denote the splitting by $L_{1} \oplus L_{2}$. Let $w_{1}$ and $w_2$ denote the weights of the $S^1$-action along $L_1$ and $L_2$ respectively.

If either $w_{1}$ or $w_{2}$ has modulus greater than $1$, then there is an isotropy $6$-manifold $Q$ such that $Q^{S^{1}}$ has two fixed components. Suppose without loss of generality that $w_{1}<-1$, so that $Q_{\max}=N$, $Q_{\min}=M_{\min}$. Then applying Theorem \ref{locbet} to the 	Hamiltonian $S^1$-action on $Q$ gives $b_{2}(Q)=b_{0}(N) =1$. Then, by Poincar\'{e} duality $b_{4}(Q)=1$, then applying Theorem \ref{locbet} again $b_{4}(Q) = b_{2}(N)=1$. Then since the symplectic form of $N$ squares positively we have that $b_{2}^{+}(N)=1$ as required. 

Hence, we may assume that $w_{1} = -1$ and $w_{2} = 1$.   By Theorem \ref{wsf}, $H(N)=0$ To ease notation we set $H(M_{\max}) = H_{\max}$ and $H_{\min} = H(M_{\min}).$ Then since $M_{\min}$ (respectively $M_{\max}$) is an isolated fixed point, all of the four weights along it are non-zero and positive (respectively negative). This implies by Theorem \ref{wsf} that $|H_{\min}| =   - H_{\min} \geq  4$ and $|H_{\max}| = H_{\max} \geq 4$. 

Considering how the tangent bundle of $M$ splits over $N$ shows that \begin{equation} \label{chernclass} [\omega]|_{N} = c_{1}(M)|_{N} = c_{1}(N) + c_{1}(L_{1}) + c_{1}(L_{2}). \end{equation}

Next, we use Lemma \ref{realblowup} to analyse the behavior of the symplectic form around $N$. Let $M_{-}$ denote the symplectic manifold provided by Lemma \ref{realblowup}.a). Then since $H(N)=0$ and $\mu_{-}$ is an equivariant symplectomorphism below level $0$ by applying Lemma \ref{limzero}. $$[\omega_{-}]|_{F_{-}} = |H_{\min}| c_{1} (H_{-}^{-1}(0)) _{F_{-}}.$$ Therefore, by Lemma \ref{realblowup}.c):

$$[\omega]|_{N} = |H_{\min}| c_{1}(L_{1}) .$$

Applying the same idea to $M_{+},\mu_{+}$ etc. give by Lemma \ref{realblowup} gives the equation:

$$[\omega]|_{N} =  |H_{\max}| c_{1}(L_{2}).$$

Substituting these equations back into equation (\ref{chernclass}) gives

$$[\omega]|_{N} = c_{1}(M)|_{N} = c_{1}(N) + (\frac{1}{|H_{\max}|} + \frac{1}{|H_{\min}|} ) [\omega]|_{N},$$
Using the fact that $|H_{\min}| \geq 4$ and $|H_{\min}| \geq 4$ implies that $$1- (\frac{1}{|H_{\max}|} + \frac{1}{|H_{\min}|}) >0 . $$ 
By, Corollary \ref{oocor} $b_{2}^{+}(N)=1$, proving the theorem.\end{proof}

The following simple example may help the reader to get an overview on the last part of the proof.

\begin{example}
Consider $\mathbb{CP}^4$, with the Fubini-Study symplectic form, rescaled so that $c_1 = [\omega]$. Consider the Hamiltonian $S^1$-action, defined for $\lambda \in S^1$ by: $$\lambda \cdot[x_{0} : x_{1} :x_{2} :x_{3} :x_{4}] =  [x_{0} : \lambda x_{1} : \lambda x_{2} : \lambda x_{3} :\lambda^2 x_{4}]. $$

Let $X_{h} \in H^{2}(\mathbb{CP}^4,\mathbb{Z})$ denote the hyperplane class i.e. the Poincar\'{e} dual of a hyperplane \footnote{We deviate from the usual notation $H$ because $H$ denotes the Hamiltonian.}. For the above action, $M_{\min}$ and $M_{\max}$ are isolated and are mapped to $-5$ and $5$ respectively by the moment map of Theorem \ref{wsf}. The only other fixed component is $N = \mathbb{CP}^2$ with $H(N)=0$ and normal bundle $\mathcal{O}(1) \oplus \mathcal{O}(1)$. Therefore $c_{1}(M)|N = c_{1}(T\mathbb{CP}^2) + 2 c_{1}(\mathcal{O}(1)) = 5X_{h}$, showing that  $c_{1}(M)|N = |H_{\max}|c_{1}(\mathcal{O}(1)) = |H_{\min}|c_{1}(\mathcal{O}(1)) $.

\end{example}

Next, we show that Theorem \ref{compin} follows from Theorem \ref{main}.

\begin{proof}[Proof of Theorem \ref{compin}]
Suppose that $(M,\omega)$ is an $8$-dimensional complete intersection with a symplectic $S^1$-action. Because $M$ is simply connected the action is Hamiltonian. By Lemma \ref{betti} $b_{2}(M)=1$. By Theorem \ref{main} the intersection form of $M$ is positive definite. Therefore by Theorem \ref{lw} $M$ is diffeomorphic to either $X_{4}(1), X_{4}(2)$ or $X_{4}(2,2)$.\end{proof}

We also present the following corollary which may be of independent interest.

\begin{corollary} \label{corctwo}
Suppose that $M$ is an $8$-dimensional closed symplectic manifold with a Hamiltonian $S^1$-action and with $b_{2}(M)=1$. Then, 

$$\int_{M} c_{2}(M)^2 \geq 0.$$
\end{corollary}

\section{Results in arbitrary dimensions} \label{arbitrary}
\subsection{Symplectic $S^1$-actions} \label{symplectic}

In the following theorem, we give a classification of the multidegrees of complete intersections admitting a Hamiltonian $S^1$-action whose fixed point sets satisfy an assumption analogous  to that of \cite[Theorem 1.1]{L}.

\begin{theorem} \label{symplcirc}
Suppose that $(M,\omega)$ is a complete intersection admitting a symplectic $S^1$-action with no fixed components $F$ satisfying $\dim(F)>0$ and $\dim(M) = \dim(F) \mod 4$.  Then $M$ is diffeomorphic to one of $X_{m}(1), X_{m}(2), X_{2m}(2,2),$  for some $m \in \mathbb{N}$ or $X_{1}(3)$, $X_{2}(3)$. 
\end{theorem}

\begin{proof}

If $\dim(M)=2$ by the classification of smooth $S^1$-actions on closed orientable surfaces the statement holds. Therefore assume $\dim(M)=2n> 2$, then since $\pi_{1}(M) = \{1\}$  by the Lefschetz hyperplane theorem, the action is Hamiltonian.

Firstly, suppose that $\dim(M)=2n$, with $n$ odd. By Lemma \ref{betti} all of the even Betti numbers of $M$ are equal to $1$ and all of the odd Betti numbers are $0$, except possibly $b_{n}(M)$. Let $F$ be a fixed submanifold. By Theorem \ref{locbet}, $b_{i}(F) \neq  0$ for at most one odd $i$. Since by assumption $\dim(F)$ is divisible by $4$, this implies by applying Poincar\'{e} duality to $F$ that all of the odd Betti numbers of $F$ vanish, because every odd Betti number appears in at least two degrees. Therefore by Theorem \ref{locbet} the odd Betti numbers of $M$ vanish. Since all of the even Betti numbers are equal to one, it follows that $\chi(M) = n+1$, hence the result follows by Theorem \ref{betichi}.

Next, suppose that $\dim(M)=2n$, with $n$ even. Then by assumption all terms on the right hand side of the equation given in Equation (\ref{equa}) of Proposition \ref{general} vanish. It follows that $$\sigma(M) = \sum_{k \in \mathbb{Z}} ( b_{4k}(M) - b_{4k+2}(M) ), $$
hence the result follows by Theorem \ref{lw}.\end{proof}

\subsection{GKM actions} \label{GKM}
In this section, we strengthen Theorem \ref{symplcirc} under the additional assumption that the action is GKM and preserves a symplectic form homotopic to the Fubini-Study form. We are able to give a complete characterization as in the classification of infinite automorphism groups due to Benoist (see Theorem \ref{benoist}).

\begin{theorem} \label{GKM}
Let $X$ be a smooth complete intersection of complex dimension greater than $2$, and let $\omega_{FS}$ be the restriction of the Fubini-Study form to $X$. Then $X$ admits a GKM torus action preserving a symplectic form $\omega$ that is homotopic to $\omega_{FS}$ $\iff$ $|Aut(X)|= \infty$.
\end{theorem} 

\begin{proof}
By Corollary \ref{completeiso} and the theorem of Benoist (see Theorem \ref{benoist}), it is sufficient to exclude the existence of such an action on $X_{2m}(2,2)$ for some positive integer $m$, set $n=2m$.  

Suppose that there is a closed symplectic manifold $(M,\omega)$ diffeomorphic to $X_{n}(2,2)$ such that $\omega$ is homotopic to $\omega_{FS}$ and such that $(M,\omega)$ has a Hamiltonian, GKM torus action.  By Lemma \ref{betti} and  \cite[Proposition 3.4]{LW}:

\begin{equation} \label{bettitwoquad}
  b_{i}(M) =    \begin{cases}
   1, \;\; \text{for $i$ even, $0 \leq i \leq 2n$, and $i \neq n$.} \\
     n+4, \;\;\text{ for } i = n. \\
      0, \;\; \text{for $i$ odd.}  
    \end{cases}   
  \end{equation}

Applying \cite[Corollary 3.1]{GS} gives $$\int_{M} c_1c_{n-1}(M) = \sum_{p=0}^{n} b_{2p}(M)(6p(p-1) + \frac{5n-3n^2}{2}).$$ Inputting Equation (\ref{bettitwoquad}) and simplifying gives $$\int_{M} c_1c_{n-1}(M)  = \frac{n}{2}(n+2)(n-1).$$

Since by assumption $n>2 $ it follows by Lemma \ref{betti} that $b_{2}(M)=1$. Hence by Lemma \ref{symfano} it is possible to rescale the hypothetical invariant symplectic form so that $c_{1}=[\omega]$.Moreover, the topological Euler characteristic of $M$ is $$2(n+2).$$ It follows that the the number of edges of the GKM graph is $$n(n+2).$$
By Lemma \ref{torone}, the collection of spheres corresponding to the edges is Poincar\'{e} dual to $c_{n-1}$.  By the adjunction formula $c_{1}$ is $(n-1)$ times the primitive generator of $H^2(X_{n}(2,2)) \cong \mathbb{Z}$. Therefore, using the condition $c_{1} = [\omega]$  the integral of $c_{1}$ on each of the spheres is a positive integer multiple of $(n-1)$. Hence the Chern number $c_1c_{n-1}$ is bounded below by $n(n+2)(n-1)$ which is a contradiction. \end{proof}

We note finally that it is possible to obtain a statement covering complex dimension $2$, at the price of requiring that the symplectic form preserved by the GKM action is monotone. Since both $X_{2}(2,2)$ and $X_{2}(3)$ admit K\"{a}hler forms preserved by a toric action, this is the strongest possible statement which can cover complex dimension $2$. Since the proof is identical to the proof of Theorem \ref{GKM}, it is omitted.

\begin{theorem}
Let $X$ be a smooth complete intersection and let $\omega_{FS}$ be the restriction of the Fubini-Study form to $X$. Then $X$ admits a GKM torus action preserving a monotone symplectic form that is homotopic to $\omega_{FS}$ $\iff$ $|Aut(X)|= \infty$.
\end{theorem}

\begin{remark} Another, perhaps more elegant but weaker, possibility is to require that the symplectic structure is symplectomorphic to one given by the Fubini-Study symplectic form. This was stated in Theorem \ref{fs} of the introduction. \end{remark}

\section{Unimodality of Betti numbers} \label{unibetti}

In this section, we will present some applications regarding the unimodality of Betti numbers of manifolds admitting a Hamiltonian $S^1$-action. Tolman has asked whether for a closed symplectic manifold with a Hamiltonian $S^1$-action, the even  Betti numbers are unimodal (i.e. weakly increasing up to the middle degree) \cite[Problem 4.3]{JHKLM}. For K\"{a}hler manifolds the property follows from the hard Lefschetz theorem.

In \cite{CK} Cho and Kim proved that  $8$-dimensional closed symplectic manifolds having a Hamiltonian $S^1$-action with isolated fixed points, the even Betti numbers are unimodal. Here we prove the same result under the weaker assumption that neither of the extremal submanifolds is $4$-dimensional.

\begin{proof}[Proof of Theorem \ref{eightunimodal}]
By Proposition \ref{general},

$$I_{JR}(M) = \sum_{\{ F \subset M^{S^1}| \dim(F)=4k>0 \}} (-1)^{\lambda_{F}} I_{JR}(F) . $$

Since by assumption $M_{\min}$ and $M_{\max}$ are not $4$-dimensional, every $4$-dimensional fixed component $F$ satisfies $\lambda_{F}=1$.

$$I_{JR}(M) = \sum_{\{ F \subset M^{S^1}| \dim(F)=4 \}} - I_{JR}(F) . $$

By Lemma \ref{dlemma} and since $F$ is a closed symplectic submanifold:
 $$I_{JR}(F) = 2b^{+}(F)-2 \geq 0,$$ hence by Poincar\'{e} duality

$$0 \geq I_{JR}(M) = \sigma(M) - 2 +2b_{2}(M) -b_{4}(M) . $$

Since $M$ is a closed symplectic manifold $ \sigma(M) >  -b_{4}(M)$. Hence 
$$b_{2}(M)  < b_{4}(M) +1, $$ and so the result follows. \end{proof}

Motivated by Theorem \ref{eightunimodal}, we state a conjecture. Following \cite{LR} a closed symplectic manifold $(M,\omega)$ is called symplectically rational if there is a connected symplectic manifold with a semi-free Hamiltonian $S^1$-action such that the reduced spaces at two regular levels are symplectomorphic to $(M,\omega)$ and $(\mathbb{CP}^n,\omega_{FS})$ respectively.

\begin{conjecture}
If $(M,\omega)$ is a closed symplectically rational $8$-manifold admitting a Hamiltonian $S^1$-action. Then, the even Betti numbers of $M$ are unimodal.
\end{conjecture}

 If a complex, projective and rational $4$-fold $M$ has such an algebraic torus action having an algebraic surface $N$ as an extremal component, then $M$ is birational to a weighted projective space bundle over $N$ via a standard GIT argument. Therefore $N$ is stably rational and hence rational \cite{Z}.  Hence $N$ contributes $0$ to the equation of Proposition \ref{general} and the proof of Theorem \ref{eightunimodal} goes through.

We also obtain a similar result to Theorem \ref{eightunimodal} for closed $12$-manifolds with second Betti number equal to $1$. To the authors knowledge this result is new, under the stronger assumption of isolated fixed points.

\begin{proposition} \label{twelve}
Suppose that $(M,\omega)$ is a closed symplectic $12$-manifold such that $b_{2}(M) = 1$. Suppose that $M$ admits a Hamiltonian $S^1$-action with no $4-$ or $8-$dimensional fixed components, then the even Betti numbers of $M$ are unimodal.
\end{proposition}
\begin{proof}
By Proposition \ref{general} we have that $I_{JR}(M) = 0$. Since  $ \sigma(M) \leq b_{6}(M),$ by Poincar\'{e} duality, $$I_{JR}(M)=0 \; \implies \; 2b_{4}(M) -b_{6}(M) \leq b_{6}(M) ,$$ and so the result follows.\end{proof}

Next we prove two inequalities on the Betti numbers of a closed symplectic manifold admitting a Hamiltonian $S^1$-action, under the assumption that there is no fixed components of positive dimension divisible by $4$. Under the assumption of isolated fixed points,  the inequalities were proved by Cho \cite[Theorem 1.2]{C}. Both inequalities may be considered as evidence for the unimodality of even Betti numbers of manifolds admitting a Hamiltonian $S^1$-action, in the sense that they would be implied by the unimodality. We formulate this as a remark after Proposition \ref{ineq}  
\begin{proposition} \label{ineq}
Suppose that $(M,\omega)$ is a closed symplectic manifold of dimension $2n$ having a Hamiltonian $S^1$-action, such that there is no fixed component of positive dimension divisible by $4$. Then, \begin{enumerate}
\item If  $\dim(M) = 0 \mod 8$  then: $$0 <  1 - b_{2}(M) +  b_{4}(M) -  \dots + b_{n}(M). $$

\item If $\dim(M) = 4 \mod 8$ then: $$ 1 - b_{2}(M) +  b_{4}(M) -  \dots - b_{n}(M)  \leq 0.$$
\end{enumerate}
\end{proposition}  
\begin{proof}

By assumption all terms on the right hand side of the equation given in Equation (\ref{equa}) of Proposition \ref{general} vanish. Thus by Proposition \ref{general} $M$ satisfies the JR equation. 

The first inequality follows from the the inequality $- b_{n}(M) < \sigma(M) $, after putting all terms to the left hand side, applying Poincar\'{e} duality and dividing by $2$. In this case the inequality is strict because the equality case is equivalent to the symplectic form being negative definite, and the top power of the symplectic form is positive.

The second inequality then follows from the inequality $\sigma(M) \leq  b_{n}(M)$, after applying Poincar\'{e} duality and dividing by $2$. \end{proof}

\begin{remark}
Suppose that $(M,\omega)$ is a closed symplectic $2n$-manifold and the even Betti numbers are unimodal. Then the inequalities of Proposition \ref{ineq} hold.
\end{remark}
To see this, suppose that $4|n$ and $b_{2i}(M) \leq b_{2i+2}(M)$ for all integers $i$ such that $2i+2 \leq n$. Then rewriting the quantity $1-b_2+b_4 - \ldots + b_{n}$, as $$1+(-b_2+b_4 )+\ldots+ (-b_{n-2} + b_{n}) ,$$ shows that it is manifestly positive. The case $n=2 \mod 4$ is proved similarly (as in Proposition \ref{ineq} the inequality is not strict in this case). We remark also that the inequalities of Proposition \ref{ineq} imply that the possibilities for any particular Betti number of the form $b_{k}$ with $k = 2n \mod 4$ are bounded, if all of the other Betti numbers are fixed, provided $k \neq n$.

We finally state a proposition regarding the Betti numbers of spin manifolds admitting a Hamiltonian $S^1$-action. 

\begin{proposition} \label{roapp}
Let $(M,\omega)$ be a closed symplectic manifold of dimension $2n$ with a Hamiltonian $S^1$-action having no fixed component of positive dimension divisible by $4$. Suppose that $n = 2 \mod 4$, that $M$ is spin, and that $b_{2i}(M)=1$ for $2i \neq n$. Then $b_{n}(M) = 2 \mod 16$.
\end{proposition}
\begin{proof}

By assumption all terms on the right hand side of the equation given in Equation (\ref{equa}) of Proposition \ref{general} vanish. Thus by Proposition \ref{general} $M$ satisfies the JR equation. By the assumptions the alternating sum of even degree Betti numbers simplifies as $2 - b_{n}(M)$. By the Rokhlin-Ochanine theorem \cite{R,O}, the signature is divisible by $16$, and so the result follows. \end{proof}

\appendix

\section{Two technical lemmas regarding symplectic reduction}
The two following technical lemmas regarding symplectic reduction are included as they are needed in the proof of Theorem \ref{main}. They are both direct consequences of known results; the Duistermaat-Heckman Theorem \cite[Theorem 1.1]{DH} and the real blow-up theorem of Guilemin and Sternberg \cite[Theorem 10.1]{GS} after some minor elaboration. 

Throughout the section, when referring to the variation of the cohomology class of the reduced symplectic form we mean in the standard way in terms of the flow of the gradient vector field as in \cite[Theorem 1.1]{DH}. To make the appendix more readable we keep this understanding without making explicit reference to it. 
\subsection{Behavior of the reduced symplectic form near extremal values}

The first lemma gives information about how the cohomology class of the reduced symplectic form behaves close to an isolated, extremal fixed point.

\begin{lemma} \label{limzero}
Let $(M,\omega)$ be a closed symplectic manifold with a Hamiltonian $S^1$-action and Hamiltonian $H$. Suppose that $M_{\min}$ is an isolated fixed point. Then, the limit of $[\omega_c]$ as $c$ approaches $H_{\min}$ from above is $0 \in H^{2}(M_{H_{\min}+\epsilon},\mathbb{R})$.
\end{lemma}
\begin{proof}
Let $w_{1}, \ldots, w_n \in \mathbb{Z} $ be the weights at $M_{\min}$. Note that $w_i > 0$ for each $i$. Consider the linear Hamiltonian $S^1$-action on $\mathbb{R}^{2n} \cong \mathbb{C}^n$ with the standard symplectic form $\tilde{\omega}$ given by, for $\lambda \in S^1$

$$\lambda \cdot (z_{1},\ldots, z_{n}) = (\lambda^{w_1}z_1 , \ldots, \lambda^{w_n}z_n ),$$ with Hamiltonian $$\mu = \sum_{i=1}^{n} w_i|z_i|^2.$$ Note that $\mu$ is a proper function, therefore for any neighborhood $U$ of $0=\mu^{-1}(0)$ there is some positive constant $\varepsilon$, such that $\mu^{-1}(x) \subset U$  for all $0< x \leq \varepsilon$.

By the equivariant symplectic neighborhood theorem \cite[Theorem A.1]{Ka} (attributed to \cite{Ma}), a neighborhood of $M_{\min} = \{ p \}$ in $M$ and a neighborhood of $0$ in $\mathbb{R}^{2n}$ are equivariantly symplectomorphic, denote by $U_1 \subset M$, $U_2 \subset \mathbb{R}^n$ and $\varphi: U_{1} \rightarrow U_{2}$ the corresponding neighborhoods and equivariant symplectomorphism. Note that since $\varphi$ is $S^1$-equivariant and $0$ is the unique fixed point in $\mathbb{R}^{2n}$ and hence in $U_{2}$, it holds that $\varphi(p) = 0$. 

 By pulling the Hamiltonian equations back along $\varphi_{*},$ it holds that $d(\mu \circ \varphi) = dH$, therefore $$\mu \circ \varphi = H + k,  $$ for some $k \in \mathbb{R}.$ Since the global minimum of $\mu$ is $0$ and global minimum of $H$ is $H(M_{\min}) := H_{\min}$, $k = -H_{\min}$.

 As was observed above, there is a possibly smaller neighborhood $U \subset U_{2}$ that is saturated with respect to $\mu$, therefore on the open subset $\varphi^{-1}(U)$, $\varphi$ maps $H^{-1}(H_{\min} + x)$ to $\mu^{-1}(x)$ by an equivariant diffeomorphism, preserving the restriction of the symplectic form to these submanifolds. This implies that the reduced spaces $M_{H_{\min}+ x}$ and $(\mathbb{R}^{2n})_{x}$ are orbifold symplectomorphic.

In particular, $$\int_{M_{(H_{\min} + x)}} [\omega_{(H_{\min} + x)}]^{n-1} = \int_{(\mathbb{R}^{2n})_{x}} [\tilde{\omega}_{x}]^{n-1} = \frac{x ^{n-1} }{\prod_{i=1}^{n} w_i}.$$
For the computation of the class of the reduced space for linear Hamiltonian $S^1$-actions of $\mathbb{R}^{2n}$ one may compute directly or see \cite{GLS2}. In particular the limit of the left hand side is zero as $x$ tends to zero from above. Since $M_{(H_{\min} + \varepsilon)}$ is orbifold diffeomorphic to a weighted projective space, $b_{2}(M_{(H_{\min} + \varepsilon)}) =1$ \cite[Theorem 1]{Kaw}, and therefore the above equation implies the required claim about the limit of the cohomology class. \end{proof}

\subsection{Behavior of the reduced symplectic form near certain fixed submanifolds}

The following other technical lemma we will need is a consequence of \cite[Theorem 10.1]{GS2}, after some minor elaboration (see the proof below). The above theorem gives a way to resolve the reduce space of a closed symplectic manifold having a simple fixed component with non-zero weights $-1,1,\ldots,1$. The following lemma roughly states that, in the most simple situation when the non-zero weights are $-1,1$ then it does not matter whether this resolution is performed via the upward or downward flowing normal bundle. This fact is alluded to in \cite{GS2} immediately below \cite[Theorem 10.2]{GS2}, we simply put it on a footing suitable for the purposes of proving Theorem \ref{main}.

\begin{lemma}{cf. \cite[Theorem 10.1, Theorem 10.2]{GS2}} \label{realblowup} 
Let $(M,\omega)$ be a closed symplectic manifold having a Hamiltonian $S^1$-action with Hamiltonian $H: M \rightarrow \mathbb{R}$. Suppose that $F$ is a fixed component of codimension $4$ with non-zero weights $-1,1$. Moreover suppose that $H(F)=0$ and $H^{-1}(0) \cap M^{S^1} = F$. Let $L_{-}$ and $L_{+}$ denote the line subbundles of the normal bundle of $F$ on which $S^1$ acts with weight $-1$ and $1$ respectively. Let $\varepsilon_0$ be chosen sufficiently small so that $H^{-1}(-\varepsilon_0,\varepsilon_0) \cap M^{S^1} = F$. 

Then, there exist symplectic manifolds $(M_{-},\omega_{-})$ and $(M_{+},\omega_{+})$ each with a Hamiltonian $S^1$-action with Hamilltonians $H_{-},H_{+}$ respectively, and $S^1$-equivariant continuous maps $\Psi_{-} :H_{-}^{-1}(-\infty,0] \rightarrow H^{-1}(-\infty,0]$ and $ \Psi_{+}: H_{+}^{-1}[0,\infty) \rightarrow H^{-1}[0,\infty)$. Such that \begin{enumerate} 

\item[a).]  The restriction of $\Psi_{-}$ (resp. $\Psi_{+}$) to  $\Psi_{-}^{-1}(H^{-1}(-\infty,0] \setminus F)$  (resp. $\Psi_{+}^{-1}(H^{-1}(-\infty,0] \setminus F)$)  is an equivariant symplectomorphism.  The restriction of $\Psi_{-}$ (resp. $\Psi_{+}$) to  $\Psi_{-}^{-1}(F)$  (resp. $\Psi_{+}^{-1}( F)$)  is a principal $S^1$-bundle.  These two principal $S^1$-bundles are isomorphic to the unit sphere bundle of $L_{-}$ and $L_{+}$ respectively. Furthermore, $H \circ \Psi_{\pm} = H_{\pm}$; in particular, $\Psi_{\pm}$ induces homeomorphisms  $$\bar{\Psi}_{\pm} : H_{\pm}^{-1}(0)/S^1 \rightarrow H^{-1}(0)/S^1.$$

 \item[b).] $H^{-1}_{\pm}(-\varepsilon_0,\varepsilon_0)$ does not contain fixed points.

\item[c).] Set $F_{\pm} := \bar{\Psi}_{\pm}^{-1}(F) /S^1.$ Then $\bar{\Psi}_{\pm}: (F_{\pm},(\omega_{\pm})_{0}) \rightarrow (F,\omega|_{F})$ is a symplectomorphism and  $$\bar{\Psi}_{\pm}^{*}(c_{1}(L_{\pm})) = e(H_{\pm}^{-1}(0))|_{F_{\pm}}.$$

  \end{enumerate}

\end{lemma}
\begin{proof}
The proof is exactly as in \cite[Theorem 10.2]{GS2}. Parts a).,b). and the second half of c). follow directly from the statement \cite[Theorem 10.2]{GS2} when one notes that the proof of \cite[Theorem 10.2]{GS2} holds whenever the action is semi-free on a neighborhood of the fixed submanifold in question, i.e. all non-zero weights are $\pm1$ (this is called quasi-free in \cite{GS2} but the term semi-free has become more commonly used). To extract the first half of c). we have to recall the proof of \cite[Theorem 10.2]{GS2}. 

Before proceeding to the proof, we consider the case when $F$ is an isolated fixed point, which contains the key idea and helps set up notation for the general case. Consider $\mathbb{C}^2$ with the standard Hamiltonian $S^1$-action with weights $-1,1$ and let $R$ be the symplectic manifold $S^1 \times \mathbb{R} \times \mathbb{C}$ with the Hamiltonian $S^1$-action $\lambda.(w,r,z') = (\lambda^{-1}w,r,\lambda z')$. The key observation of the real blow-up construction is that there is an $S^1$-equivariant map $$ \psi : \{(w,r,z') \in R \mid r \leq 0\} \rightarrow \mathbb{C}^2  $$ which restricts to the subset $  \{(w,r,z') \in R \mid r <  0\} $ as an equivariant symplectomorphism with image $\{ (z_1,z_2) \in \mathbb{C}^2 \mid z_1 \neq 0  \}$ \footnote{The map is given by polar coordinates with a suitable radial factor, see \cite[Section 10]{GS} for the explicit formula.}.  Therefore we can resolve the fixed point of the action on $\mathbb{C}^2$ by gluing these neighborhoods. In the case where $F$ is a submanifold of positive dimension the construction may be globalized using the the machinery of principal bundles.

Let $N = L_{-} \oplus L_{+}$ be the equivariant splitting of the normal bundle of $F$, and $S_{-},S_{+}$ the unit sphere bundles of $L_{-},L_{+}$. Define a bundle of unitary frames of $N$ respecting the splitting, $P|_{p} = \{(v_1,v_{2}) \mid v_1 \in  L_{-}|_{p}, v_{2} \in L_{+}|_{p} \}$. Then $P$ is an $S^1 \times S^1$-principal bundle. Let  $\mathfrak{t}$ denote the Lie algebra of $S^1 \times S^1$. The classical fact that the cotangent bundle of a Lie group is trivial implies that the the vertical cotangent bundle of $V^*(P)$ is isomorphic to  $$T^{*}(V^*(P)) \cong V^*(P) \times \mathfrak{t}.$$ Let $\omega_{cot}$ denote the natural symplectic form on $T^{*}(P)$, one may also verify the action of $S^1 \times S^1$ is Hamiltonian with respect to $\omega_{cot}$ and with moment map $\mu_1$ equal to the projection to the $\mathfrak{t}$ factor.

 The choice of a connection on $P$, gives rise to a natural inclusion $i$ of the vertical cotangent bundle $V^{*}P$ into $T^{*}(P)$, and $\pi^{*}\omega_{F} + i^{*} \omega_{cot}$ is an invariant symplectic form on a neighborhood of the zero section of $V^*(P).$ 

Let $S^1 \times S^1$ act on $V^*(P) \times \mathbb{C}^2$ by $g \cdot (p_1,p_{2}) = (gp_1,g^{-1}p_{2})$. This action is Hamiltonian with moment map $\mu_{1} - \mu_{2}$, where $\mu_2$ is the moment map for the standard toric action on $\mathbb{C}^2$. Via that principal bundle-vector bundle correspondence, it may be shown that $\mu^{-1}(0)/S^1 \times S^1$ is diffeomorphic to the total space of $N$.

 To see this, define a map $m: P \times \mathbb{C}^2 \rightarrow \mu^{-1}(0)$, $$m(p,v) = (p,-\mu_{2}(v), v).$$ Here we are using the fact that the vertical cotangent bundle is canonically trivial, as explained above. Then it may be checked that $m$ is a $S^1 \times S^1$-equivariant diffeomorphism, then the claim follows from standard application of the principal bundle-vector bundle correspondence that $P \times \mathbb{C}^2/S^1\times S^1$ and  hence $\mu^{-1}(0)/S^1 \times S^1$ are diffeomorphic to the total space of $N$. For the proof of the first half of c). it is helpful to identify the preimage of $F$ in $\mu^{-1}(0)$ by the quotient map $\pi$ (more precisely the composition of two quotient maps: the $S^1 \times S^1$-action to the obtain the local model around $F$ and the $S^1$-quotient to obtain the reduced space). It is simply $$ \pi^{-1}(F) = P \times \{0\} \times \{0\} \subset \mu^{-1}(0) .$$ 

As in the linear version of real blow-up we replace $V^*(P) \times \mathbb{C}^2$ by $V^*(P) \times R$ with corresponding moment map $\mu'$.  The real blow up is defined as  $\mu'^{-1}(0)/S^1 \times S^1$ and the map $\Psi_-$ is induced by the map $Id \times \psi$ in the quotient (as above this means quotienting by $S^1 \times S^1$ to the local model and then by the $S^1$-action to the reduced space, we denote the composition of these quotient maps $\pi'$). It follows that: $$ \pi'^{-1}( \Psi_{-}^{-1} (F) ) =  P \times \{0\}  \times S^1 \subset \mu'^{-1}(0) ,$$ and  $\bar{\Psi}_{-}$ is induced over $F$ by the map $f: \pi'^{-1}( \Psi_{-}^{-1} (F) ) \rightarrow  \pi^{-1}(F)$ defined by $$(p,0,z) \mapsto (p,0,0) .$$  Finally, by direct calculation it is immediate to verify that $f$ induces a symplectomorphism over $F$.
\end{proof}

Departamento di Mathematica e Informatica ``U. Dini", Universit\`{a} Di  Firenze.\\

\noindent nicholas.lindsay@unifi.it.

\end{document}